\documentclass[]{amsart}
\usepackage[utf8]{inputenc}
\usepackage{amsmath,amssymb,amsthm,mathtools}
\usepackage{graphicx}
\usepackage{verbatim}
\usepackage{xcolor}

\newtheorem{defn}{Definition}

\newtheorem{prop}{Proposition}
\newtheorem{lem}{Lemma}
\newtheorem{remark}{Remark}

\newtheorem*{property*}{Property}

\theoremstyle{remark}
\newtheorem*{sketch*}{Sketch of Proof}

\title{A note concerning the invertibility of certain alternant matrices}

\author{Jeff Ledford}
\address{Longwood University, Farmville, VA 23909, USA}
\email{ledfordjp@longwood.edu}
\thanks{This work was supported by the PRISM program at Longwood University.}

\begin{document}

\maketitle

\section{Introduction}

This short note details an elementary method to show that certain alternant matrices are invertible.  An alternant matrix takes the form $[ f_i(x_j)] $ for some sequence of functions $(f_i:1\leq i\leq N)$ and domain values $(x_j:1\leq j\leq N)$.  The classical example of such a matrix is the Vandermonde matrix 
\[
V_N:=\begin{bmatrix}
x_{i}^{j-1}
\end{bmatrix}_{1\leq i,j\leq N}.
\]
The result of interest here is that $V_N$ is invertible whenever $(x_j:1\leq j\leq N)$ consists of distinct points.  Fixing $N\in\mathbb{N}$, then we can view this result as following from the fundamental theorem of algebra.  Indeed, the homogeneous system $V_N\mathbf{a}= \mathbf{0}$ yields $[f_{\mathbf{a}}(x_i)]=\mathbf{0}$ for some $f_{\mathbf{a}}\in\Pi_{N-1}$.  Hence the only solution is the trivial solution $\mathbf{a}=\mathbf{0}$.  

In the next section we will (slightly) expand upon this idea and introduce notation.  The final section contains examples.

\section{General Alternant Systems}

Suppose we have a set of continuous functions $G:=\{g_1,g_2,\dots,g_N  \},$
where for $1\leq j\leq N$, $g_j:I\to\mathbb{R}$ for some interval $I\subset\mathbb{R}$.  Let $\mathcal{G}:=\text{span}(G)$
and for $f\in\mathcal{G}$, let $f^{\sharp}$ denote the number of roots that $f$ has on $I$, and $\mathcal{G}^\sharp=\sup_{f\in \mathcal{G}\setminus\{\mathbf{0}\}}f^{\sharp}$.
Our first result is straightforward.

\begin{lem}\label{alternant_inverse}
Let $N\in\mathbb{N}$ and suppose that $G$ satisfies $\mathcal{G}^\sharp<N$ and $X=(x_i:1\leq i\leq N)\subset I$ consists of $N$ distinct points.  Then the alternant matrix
\[
A(G,X):=\left[ g_j(x_i)  \right]_{1\leq i,j\leq N}
\]
is invertible.
\end{lem}

\begin{proof}
Consider the product $A(G,X)\mathbf{a}$ in the variable $\mathbf{a}$.  This results in the vector $\mathbf{v}\in\mathbb{R}^N$, whose $i$-th component is given by $f(x_i)$, where $f\in\mathcal{G}$.  Now suppose that $A(G,X)$ is non-invertible.  Then the homogeneous system $A(G,X)\mathbf{a}=\mathbf{0}$ has a non trivial solution $\mathbf{a}_0$, which leads to $f_0\in\mathcal{G}$ that has $N$ roots on $I$.  This contradicts the fact that $\mathcal{M}<N$, which shows that $A(G,X)$ must be invertible.  
\end{proof}

\begin{lem}\label{derivative_roots}
Suppose that $f\in C^1(I)$ and that $f'$ has $n$ distinct roots in $I$.  Then $f$ has at most $n+1$ roots in $I$.
\end{lem}

\begin{proof}
We partition $I$ into $n+1$ subintervals with the roots of $f'$ as endpoints.  Since $f\in C^1(I)$, $f$ is monotone on each subinterval, so that there are at most $n+1$ roots of $f$.  
\end{proof}

\section{Examples}

\subsection{power functions}

For a fixed $r\in\mathbb{R}$, consider the set of functions
\[
G_N:=\{x^{r+j-1}:1 \leq j \leq N \},
\]
defined on the interval $(0,\infty)$.  If $f\in\mathcal{G}_N\setminus\{\mathbf{0}\}$, we have $f(x)=x^{r}\tilde{f}(x)$, where $\tilde{f}\in \Pi_{N-1}$.  Thus $f$ has at most $N-1$ roots, hence $\mathcal{G}_{N}^\sharp\leq N-1$.  To see that $\mathcal{G}_{N}^\sharp= N-1$, we let $\tilde{f}=(x-x_1)(x-x_2)\cdots(x-x_{N-1})$  

\begin{remark}
The Vandermonde system corresponding to $V_N$ above is the special case $r=0$.
\end{remark}

\subsection{logarithms and polynomials}
Throughout this subsection, we set 
\[
\mathcal{H}_N:=\{\ln(x)p(x)+q(x): p,q\in\Pi_{N-1} \},
\]
and restrict our attention to the interval $x\in (1,\infty)$.  We note that $|H_N|=2N$.  We begin with two derivative formulas.

\begin{lem}\label{log_derivative_1}
For $k\in\mathbb{N}_0$, we have 
\[
D^{k+1}\left( x^k\ln(x) \right) = k!x^{-1}.
\]
\end{lem}

\begin{proof}
We induct on $k\in\mathbb{N}_0$.  For $k=0$, we just get the familiar derivative formula for the logarithm
\[
\left( \ln(x) \right)' = \dfrac{1}{x}.
\]
Now assume that the formula holds for some $k\geq 0$.  We have
\begin{align*}
   D^{k+2}\left( x^{k+1}\ln(x) \right) &=  D^{k+1} \dfrac{{\rm d}}{{\rm d}x}\left(x (x^k \ln(x) )\right) \\
    &=D^{k+1} \left( x^k\ln(x) + x(kx^{k-1}\ln(x)+x^{k-1}) \right)\\
    &=D^{k+1}\left( (k+1)x^k\ln(x) +x^k\right)\\
    &=(k+1)!x^{-1},
\end{align*}
as desired.
\end{proof}

\begin{lem}\label{log_derivative_2}
Let $N\geq 2$ and suppose that $p\in\Pi_{N-1}$, with 
\[
p(x)=\sum_{k=0}^{N-1}a_k x^k.
\]
Then
\[
D^{N} \left( p(x)\ln(x) \right) = x^{-N}\sum_{j=0}^{N-1}(-1)^{N-1+j}c_j a_j x^{j},
\]
for some positive constants $c_j$.
\end{lem}

\begin{proof}
We induct on $N\geq 2$.  Two applications of the product rule yields the base case:
\[
\left( (ax+b)\ln(x)\right)'' = \dfrac{ax-b}{x^2}.
\]
Now we assume that the conclusion holds for all $k$ with $2\leq k\leq N$.  Consider $p\in\Pi_{N}$. We have
\[
p(x)\ln(x)=\left(a_{N}x^{N}+q(x)\right)\ln(x),
\]
so that
\begin{align*}
D^{N+1} \left( p(x)\ln(x) \right)  &= a_N D^{N} \left( x^N\ln(x) \right) + D^{N+1} \left( q(x)\ln(x) \right) \\
&=a_N D^{N+1} \left( x^N\ln(x) \right) + D\left(x^{-N}\sum_{j=0}^{N-1}(-1)^{N-1+j}c_j a_j x^{j}\right)\\
&=N!a_Nx^{-1}+\sum_{j=0}^{N-1}(-1)^{N-1+j}c_j(j-N) a_j x^{j-N-1}\\
&=N!a_Nx^{-1}+\sum_{j=0}^{N-1}(-1)^{N+j}c_j(N-j) a_j x^{j-N-1}\\
&=x^{-N-1}\left( N!a_Nx^{N}+\sum_{j=0}^{N-1}(-1)^{N+j}c_j(N-j) a_j x^{j-N-1}  \right)\\
&=x^{-N-1}\sum_{j=0}^{N}(-1)^{N+j}\tilde{c}_ja_j x^j
\end{align*}
We've used Lemma \ref{log_derivative_1} in the third line.  The result follows from the fact that $c_j>0$ and $N-j>0$, so that $\tilde{c}_j>0$. 
\end{proof}

\begin{remark}
The point of this calculation is to show that the polynomial in question is an \emph{alternating} combination of the original.
\end{remark}

\begin{lem}
If $f\in\mathcal{H}_N\setminus\{\mathbf{0} \}$, then $f$ has at most $2N-1$ roots.  That is, $\mathcal{H}_N^\sharp\leq 2N-1$.
\end{lem}

\begin{proof}
If $f\in\mathcal{H}_N\setminus\{\mathbf{0} \}$ with 
\[
f(x)= \ln(x)\sum_{k=0}^{N-1}a_k x^k + \sum_{m=0}^{N-1}b_m x^m,
\]
then Lemma \ref{log_derivative_2} provides 
\[
D^N f(x)=x^{-N}\sum_{j=0}^{N-1}(-1)^{N-1+j}c_j a_j x^{j}.
\]
Since $x^{-N}>0$ on $(1,\infty)$, $f^{(N)}$ has at most $N-1$ roots.  Since $f\in C^{\infty}(1,\infty)$ we can use Lemma \ref{derivative_roots} repeatedly to conclude that $f$ has at most $2N-1$ roots. 
\end{proof}

\subsection{a general example}

In this section, we look for conditions on a collection of functions such that that guarantee at most a specified number of roots.  This problem has been studied before, see \cite{Haukkanen} for a recent example.  Our example arises from a problem in approximation theory, as a result it is more specialized.  Our collection takes the form  
\[
\mathcal{F}_{m,n}:=\{p(x)+F(x)q(x): p\in\Pi_m,q\in \Pi_{n} \},
\]
and we seek conditions on a function $F$ such that $\mathcal{F}_{m,n}^{\sharp}<m+n+2$.  To this end, we introduce the notion of compatibility.

\begin{defn}
For  $k,n,l\in\mathbb{N}_0$, we say the function $F$ is \emph{$k$-compatible with $\Pi_n$ of degree $l$} on an interval $I$ if there exists a function $\tilde{F}$ (depending on $k$ and $n$) such that
\begin{enumerate}
    \item $F\in C^{k+1}(I)$, 
    \item for all $q\in \Pi_n$, $D^k\left( Fq  \right) = \tilde{F}\tilde{q}$,
    \item $\tilde{F}$ is monotone on $I$ and $\tilde{F}(x)\neq 0$, and
    \item $\tilde{q}$ has at most $l$ roots in $I$.\\
\end{enumerate}

\end{defn}

\begin{remark}
This definition is motivated by the example in the previous section, which shows that $F(x)=\ln(x)$ is $N$-compatible with $\Pi_{N-1}$ of degree $N-1$, with $\tilde{F}(x)=x^{-N}$ on $(1,\infty)$ and $\tilde{q}\in\Pi_{N-1}$.
\end{remark}

\begin{remark}
The fourth condition gives us a bit of flexibility in counting the roots.  For instance if $\tilde{q}\in\Pi$, we could appeal to Descartes' rule of signs or the Budan-Fourier theorem to count the roots.
\end{remark}

\begin{prop}
Let $k,l,m,n\in \mathbb{N}_0$ and suppose that $F$ is $k$-compatible with $\Pi_n$ of degree $l$ on $I$.  If $k>m$, then $\mathcal{F}_{m,n}^{\sharp}\leq k+l$.  If, additionally, $l< n-(k-m)+2$, then $\mathcal{F}_{m,n}^{\sharp}< m+n+2$.
\end{prop}
\begin{proof}
Since $k>m$, $D^k\left(p(x)+F(x)q(x) \right) = D^k\left( F(x)q(x)\right) = \tilde{F}(x)\tilde{q}(x)$, which has at most $l$ roots.  Applying Lemma \ref{derivative_roots} repeatedly gives us that $p(x)+F(x)q(x)$ has at most $k+l$ roots, hence $\mathcal{F}_{m,n}^{\sharp}\leq k+l$.  The additional assumption $l<n-(k-m)+2$ is equivalent to $k+l<m+n+2$.
\end{proof}

We conclude this section with a general version of the example above.  Let $i,m,n\in\mathbb{N}_{0}$, $I:=(1,\infty)$, and define
\[
F_{i,m,n}:=\left\{ p(x)+x^i\ln(x)q(x): p\in\Pi_m,q\in \Pi_{n} \right\}.
\]
Lemmas \ref{log_derivative_1} and \ref{log_derivative_2} provide
\[
D^{n+i+1}(x^i\ln(x)q(x)) =  x^{-(n+1)}\sum_{j=0}^{n}(-1)^{n+j}c_{i+j} a_{i+j} x^{j},
\]
where $q(x)=\sum_{j=0}^{n}a_{i+j}x^{i+j}$ and $c_{i+j}>0$.  So if all of the coefficients $a_{i+j}>0$ or $a_{i+j}<0$ Descartes' rule of signs gives that $x^i\ln(x)$ is $(n+i+1)$-compatible with $\Pi_n$ of degree $n$. The proposition provides that if $m< n+i+1 < m+2$, then $\mathcal{F}_{m,n}^{\sharp}< m+n+2$.  This condition reduces to $m=n+i$.    However, we can improve this result using the following lemma.

\begin{lem}\label{log_derivative_3}
For $k\in\mathbb{N}$, we have 
\[
D^{k}\left( x^k\ln(x) \right) = k!\ln(x)+C_k,
\]
for some positive constant $C_k$.
\end{lem}
\begin{proof}
We induct on $k\in\mathbb{N}$.  The result for $k=1$ follows from the Leibniz rule
\[
\left(x\ln(x)\right)'=\ln(x)+1.
\]  
Now suppose that the formula holds for some $k\in\mathbb{N}$.  Consider
\begin{align*}
D^{k+1}\left( x^{k+1}\ln(x)\right)&= D^{k}\left( D\left(xx^k\ln(x)\right)\right)     \\
&=D^{k}\left((k+1)x^k\ln(x)+ x^k \right)\\
&=(k+1)!\ln(x)+(k+1)C_k+k!\\
&=:(k+1)!\ln(x)+\tilde{C}_{k+1}.
\end{align*}
\end{proof}
Hence if $m<n+i$, we have
\begin{align*}
D^{n+i}\left( p(x)+x^i\ln(x)q(x)\right) &=   D^{n+i}\left( x^i\ln(x)q(x)\right)\\
&= (n+i)!\ln(x)+C_{n+i} +D^{n+i}\left( x^i\ln(x)w(x) \right),
\end{align*}
where $w(x)=q(x)-a_nx^n\in \Pi_{n-1}$.  Now Lemmas  \ref{log_derivative_1} and \ref{log_derivative_2} give us
\begin{align*}
D^{n+i}\left( p(x)+x^i\ln(x)q(x)\right)&= (n+i)!\ln(x)+C_{n+i} +D^{n+i}\left( x^i\ln(x)w(x) \right)\\
&=(n+i)!\ln(x)+C_{n+i} + x^{-n}\sum_{j=0}^{n-1}(-1)^{n+j-1}c_{i+j}a_{i+j}x^{j}\\
&=:x^{-n}G(x),
\end{align*}
where 
\[
G(x):= (n+i)!x^n\ln(x)+C_{n+i}x^n + \sum_{j=0}^{n-1}(-1)^{n+j-1}c_{i+j}a_{i+j}x^{j} .
\]
\begin{remark}
Note that the two leading terms in the polynomial part of $G$ share the same sign.
\end{remark}
Now Lemma \ref{log_derivative_3} shows that for $x\in I$
\[
D^nG(x) = (n+i)!n!\ln(x)+C_n+n!C_{n+i}>0,
\]
hence $p(x)+x^i\ln(x)q(x)$ has at most $2n+i$ roots provided $m<n+i$.  So to make sure $\mathcal{F}_{m,n}^{\sharp}<m+n+2$, we must have $m<n+i<m+2$, that is
$m=n+i-1$.

\subsection{$m>n$}
Consider the collection
\[
\mathcal{F}_{m,n}:=\{p(x)+\ln(x)q(x): p\in\Pi_m,q\in \Pi_{n} \},
\]
where $m>n$.  Since $\ln(x)$ is $n+1$ compatible with $\Pi_n$ of degree $n$ with $\widetilde{\ln}(x)=x^{-n-1}$, $f\in \mathcal{F}_{m,n}$ is at most $m+n+1$ to $1$.  Thus $\mathcal{F}_{m,n}^{\sharp}\leq m+n+1$.

\end{document}